\documentclass[12pt]{amsart}
\usepackage[alphabetic]{amsrefs}\allowdisplaybreaks
\usepackage{mathdots, blkarray}
\usepackage{multicol}
\usepackage{graphicx}
\usepackage{amscd}
\usepackage{upgreek}
\usepackage{stmaryrd}
\usepackage{longtable}
\usepackage[T1]{fontenc}
\usepackage{latexsym, amsmath, amssymb, amsthm, mathrsfs}
\usepackage[svgnames]{xcolor}
\usepackage{rsfso}
\usepackage[mathscr]{eucal}
\usepackage{mathtools}
\usepackage{mathptmx}
\usepackage{titletoc}
\usepackage{wrapfig}
\usepackage{float}
\usepackage{xypic}
\usepackage{microtype}
\usepackage{dsfont}
\usepackage{xcolor}
\usepackage{color}
\usepackage[colorlinks = true,
            linkcolor  = DarkBlue,
            urlcolor   = DarkRed,
            citecolor  = DarkGreen]{hyperref}
\usepackage{hyperref}
\allowdisplaybreaks	
\usepackage[centering, includeheadfoot, hmargin=1.0in, tmargin=0.5in,
  bmargin=0.6in, headheight=6pt]{geometry}
\newtheorem{theorem}{Theorem}[section]

\newtheorem{lem}[theorem]{Lemma}

\newtheorem{thm}[theorem]{Theorem}

\newtheorem{exam}[theorem]{\rm\textsc{Example}}

\newcommand{\ideal}[1]{\ensuremath{\left\langle #1 \right\rangle}}

\DeclareMathOperator{\SL}{SL}
\DeclareMathOperator{\dia}{diag}

\DeclareMathOperator{\GL}{GL}

\DeclareMathOperator{\Tr}{Tr}
\DeclareMathOperator{\Jac}{Jac}
\DeclareMathOperator{\Ind}{Ind}
\DeclareMathOperator{\LT}{LT}

\newcommand{\A}{\mathcal{A}}
\newcommand{\B}{\mathcal{B}}
\newcommand{\J}{\mathfrak{J}}

\newcommand{\C}{\mathcal{C}}

\newcommand{\N}{\mathbb{N}}
\newcommand{\F}{\mathbb{F}}
\newcommand{\R}{\mathcal{R}}

\newcommand{\ra}{\longrightarrow}

\newcommand{\ol}{\overline}

\newcommand{\last}[1]{\begin{flalign*} && #1&&\Diamond \end{flalign*}}
\begin{document}
\title{Some four-dimensional orthogonal invariants}
\def\shorttitle{Some four-dimensional orthogonal invariants}

\author{Shan Ren}

\address{School of Mathematics and Statistics, Changchun University, Changchun 130022, China}
\email{rens734@nenu.edu.cn}

\author{Runxuan Zhang}
\address{Department of Mathematics and Information Technology, Concordia University of Edmonton, Edmonton, AB, Canada, T5B 4E4}
\email{runxuan.zhang@concordia.ab.ca}

\begin{abstract}
Let $p$ be an odd prime and $\F_p$ be the prime field of order $p$. Consider 
a $2$-dimensional orthogonal group $G$ over $\F_p$ acting on the standard representation $V$ and the dual space $V^*$.  We compute the invariant ring $\F_p[V\oplus V^*]^G$ via explicitly exhibiting a minimal generating set. Our method
provides an application of the $s$-invariants appearing in the covariant theory of finite groups. 
\end{abstract}

\date{\today}
\maketitle \baselineskip=16pt

\dottedcontents{section}[1.16cm]{}{1.8em}{5pt}
\dottedcontents{subsection}[2.00cm]{}{2.7em}{5pt}

\section{Introduction}
\setcounter{equation}{0}
\renewcommand{\theequation}
{1.\arabic{equation}}
\setcounter{theorem}{0}
\renewcommand{\thetheorem}
{1.\arabic{theorem}}

\noindent  Let $k$ be a field of any characteristic, $G$ a finite group and 
$V$ be a faithful finite-dimensional representation of $G$ over $k$. The action of $G$
on the dual space $V^*$ can be extended algebraically to a $k$-linear action of $G$ on the symmetric algebra $k[V]$ on $V^*$, i.e., elements of $G$ can be viewed as $k$-algebraic automorphisms of 
$k[V]$. Choosing a basis $\{x_1,\dots,x_n\}$ for $V^*$, we may identify $k[V]$ with the polynomial ring $k[x_1,\dots,x_n]$.
The invariant ring
$$k[V]^G:=\{f\in k[V]\mid \upsigma\cdot f= f,\textrm{ for all }\upsigma\in G\}$$
consisting of all polynomials fixed by all elements of $G$, is the main object of study in algebraic invariant theory; see for example, \cite{CW11, DK15}, and \cite{NS02} for general references to the invariant theory of finite groups. 

The invariant theory of classical groups over finite fields, originating from the classical Dickson invariants \cite{Dic11}, has  substantial applications in algebraic topology and commutative algebra, and has occupied a central position in modular invariant theory; see \cite{CSW24, CSW25} and \cite{CSW21} for the recent development in computing modular invariants of finite classical groups acting on their standard representations. Roughly speaking, classical groups over finite fields can be divided into three families: symplectic, unitary, and orthogonal groups; see \cite{Tay92} or \cite{Wan93}. Compared with the cases of finite symplectic and unitary groups, the invariant theory and geometry for finite orthogonal groups would be relatively more complicated. 

Let $O_n(q)$ be an $n$-dimensional orthogonal group over a finite field $\F_q$ acting on its standard representation $V$ and the dual representation $V^*$. Computing the invariant ring $\F_q[mV\oplus rV^*]^{O_n(q)}$ of $m$ vectors and $r$ covectors has been a difficult task in algebraic invariant theory even for the case $(m,r)=(1,0)$. 
Based on several earlier studies on the calculations of finite orthogonal invariants \cite{CK92, TW06, Chu01} and \cite{FF17}, 
Campbell, Shank, and Wehlau recently have made important progress in computing $\F_q[V]^{O_n(q)}$ in \cite[Theorem 4.6]{CSW24}, demonstrating that the invariant ring $\F_q[V]^{O_n(q)}$ is a complete intersection when $O_n(q)=O^+_n(\F_q)$ denotes the finite orthogonal group of plus type in odd characteristic.
More progress on the case $(m,r)=(1,1)$, i.e.,  modular invariants of one vector and one covector for other finite classical groups can be found in \cite{BK11, Che14}, and \cite{Ren24}.  Also, see \cite{CW19, CT19} and \cite{HZ20} for some calculations on modular invariant fields of several vectors and covectors.

In this article, we are interested in computing the invariants of one vector and one covector for finite two-dimensional orthogonal groups. More precisely, we will focus on
$$\F_p[V\oplus V^*]^{O_2(p)}$$
where $p$ denotes an odd prime and $\F_p$ denotes the prime field of order $p$. Note that two related works  
are available to show the difficulties of computing the invariants of finite 2-dimensional orthogonal groups:  \cite{Che18} for 
vector modular invariants of $O_2(q)$ in even characteristic and \cite{LM24} for separating vector invariants of $O_2(q)$ in odd characteristic. 

We denote by $\SL_n(\F_p)$ and $\GL_n(\F_p)$ the special linear group and the general linear group over $\F_p$, respectively. To articulate our main results, we suppose that $p>2$ and $Q\in\F_p[y_1,y_2]$ denotes a non-degenerate quadratic form over $\F_p$. Up to equivalence, it is well-known that there are two canonical quadratic forms 
\begin{equation}
\label{ }
Q_+=y_1y_2\textrm{ and }Q_-=y_1^2-\uplambda\cdot  y_2^2,
\end{equation}
where $\uplambda\in\F_p^\times$ denotes a non-square element; see \cite[Section 7.4]{NS02}. Equivalent quadratic forms correspond to isomorphic orthogonal groups. Thus two orthogonal groups, denoted as $O_2^+(\F_p)$ and $O_2^-(\F_p)$, are defined as the stabilizers of $Q_+$ and $Q_-$ in $\GL_2(\F_p)$, respectively.
Let $V$ be the standard 2-dimensional representation of $\GL_2(\F_p)$ with a basis $\{y_1,y_2\}$ and choose $\{x_1,x_2\}$ as a basis of $V^*$ dual to $\{y_1,y_2\}$. We identify $\F_p[V\oplus V^*]$ with $\F_p[x_1,x_2,y_1,y_2]$ and we would like to compute $\F_p[V\oplus V^*]^{O_2(p)}=\F_p[x_1,x_2,y_1,y_2]^{O_2(p)}$, where $O_2(p)=O_2^+(\F_p)$ and $O_2^-(\F_p)$.

To compute $\F_p[V\oplus V^*]^{O_2^+(\F_p)}$, we consider the $2$-dimensional special orthogonal group:
\begin{equation}
\label{ }
SO_{2}^+(\F_p):=O_2^+(\F_p)\cap \SL_2(\F_p).
\end{equation}
The first result computes $\F_p[V\oplus V^*]^{SO_2^+(\F_p)}$ as follows.

\begin{thm}\label{thm1}
The invariant ring $\F_p[V\oplus V^*]^{SO_{2}^+(\F_p)}$ is generated by
\begin{equation}
\label{eq-thm1}
\A:=\left\{x_1x_2, y_1y_2, x_1y_1, x_2y_2, x_1^{p-1-i}y_2^i, x_2^{p-1-i}y_1^i\mid 0\leqslant i \leqslant p-1\right\}.
\end{equation}
\end{thm}

Together with the relative Reynolds operator, we may use Theorem \ref{thm1} to prove the following second result that 
computes $\F_p[V\oplus V^*]^{O_2^+(\F_p)}$.

\begin{thm}\label{thm2}
The invariant ring $\F_p[V\oplus V^*]^{O_2^+(\F_p)}$ is generated by
\begin{equation}
\label{eq-thm2}
\B:=\left\{x_1x_2, y_1y_2, x_1y_1+x_2y_2, x_1^{p-1-i}y_2^i+x_2^{p-1-i}y_1^i\mid 0\leqslant i\leqslant p-1\right\}.
\end{equation}
\end{thm}

The structure of $\F_p[V\oplus V^*]^{O_2^-(\F_p)}$ is more complicated than that of $\F_p[V\oplus V^*]^{O_2^+(\F_p)}$.
Magma calculation \cite{BCP97} suggests that the cardinality of a generating set of $\F_p[V\oplus V^*]^{O_2^-(\F_p)}$ would become larger and larger as $p$ increases. This also means that finding a pattern revealing 
generating relations of $\F_p[V\oplus V^*]^{O_2^-(\F_p)}$ might be impossible. 
We use the Jacobian criterion appearing in the covariant theory of finite groups (see \cite[Theorem 3]{BC10}), compute the corresponding $s$-invariant, and eventually find a free basis of $\F_p[V\oplus V^*]^{O_2^-(\F_p)}$ over
$\F_p[V\oplus V^*]^{O_2^-(\F_p)\times O_2^-(\F_p) }$ in Theorem \ref{thm4}. By this free basis, we may obtain the following third result, which is a direct consequence of Theorem \ref{thm4}.

\begin{thm}\label{thm3}
The invariant ring $\F_p[V\oplus V^*]^{O_2^-(\F_p)}$ is generated by
\begin{equation}
\label{eq-thm3}
\C:=\left\{\begin{aligned}
&x_1^2-\lambda x_2^2, x_1^{p+1}-\lambda x_2^{p+1}, y_1^2 - \lambda^{-1} y_2^2, y_1^{p+1}-\lambda^{-1} y_2^{p+1},\\
&x_1y_1+x_2 y_2, \Tr(x_1^{p+1-i}y_1^i) \mid 1\leqslant i\leqslant p
\end{aligned}
\right\}
\end{equation}
where $\uplambda=-1$ if $p \equiv 3 \mod 4$; and $\uplambda$ generates $\mathbb{F}_p^{\times}$ if $p \equiv 1 \mod 4$.
\end{thm}

Note that the advantage of our method in Theorem \ref{thm4} is that we avoid seeking generating relations of $\F_p[V\oplus V^*]^{O_2^-(\F_p)}$. In some traditional methods, determining some generating relations is difficult (even for some 4-dimensional invariants) but very helpful in computing generating sets or free bases of an invariant ring; see for example, \cite{CH96} and \cite{Che21}.

We close this section by presenting several examples for small prime $p$.

\begin{exam}{\rm
(1) Consider $p=3$. The invariant ring $\F_3[V\oplus V^*]^{O_2^+(\F_3)}$ is a complete intersection, generated by the following invariants:
\begin{eqnarray*}
f_1:=x_1x_2, f_2:=x_1^2+x_2^2, f_3:=y_1y_2, f_4:=y_1^2+y_2^2, u:=x_1y_1+x_2y_2, v:=x_1y_2+x_2y_1
\end{eqnarray*} 
subject to the two relations: $f_1\cdot f_4+f_2\cdot f_3-u\cdot v=0$ and 
$f_1\cdot f_3+f_2\cdot f_4-u^2-v^2=0.$

(2) Suppose that $p=5$. The invariant ring $\F_5[V\oplus V^*]^{O_2^-(\F_5)}$ is generated by the primary invariants $$\left\{x_1^2+2\cdot x_2^2, x_1^6+2\cdot x_2^6, y_1^2+3\cdot y_2^2, y_1^6+3\cdot y_2^6\right\}$$
and six secondary invariants 
$\left\{x_1y_1+x_2 y_2, \Tr(x_1^5y_1), \Tr(x_1^4y_1^2), \Tr(x_1^3y_1^3), \Tr(x_1^2y_1^4), \Tr(x_1y_1^5)\right\}.$

(3) If $p=7$, then $\F_7[V\oplus V^*]^{O_2^-(\F_7)}$ is generated by the
four primary invariants $$\left\{x_1^2+x_2^2, x_1^8+x_2^8, y_1^2+y_2^2, y_1^8+y_2^8\right\}$$
together with the eight secondary invariants
\last{\left\{x_1y_1+x_2y_2, \Tr(x_1^7y_1), \Tr(x_1^6y_1^2), \Tr(x_1^5y_1^3), \Tr(x_1^4y_1^4), \Tr(x_1^3y_1^5),\Tr(x_1^2y_1^6), \Tr(x_1y_1^7)\right\}.}
}\end{exam}

\section{$O_2^+(\F_p)$-Invariants} \label{sec2}
\setcounter{equation}{0}
\renewcommand{\theequation}
{2.\arabic{equation}}
\setcounter{theorem}{0}
\renewcommand{\thetheorem}
{2.\arabic{theorem}}

\noindent The main purpose of this section is to prove Theorem \ref{thm2}, calculating the invariants of one vector and one covector of $O_2^+(\F_p)$. Let us begin by recalling some fundamentals about $O_2^+(\F_p)$ and its invariants. 
Note that $p\geqslant 3$ and $|O_2^+(\F_p)|=2(p-1)$. Thus the standard representation $V$ is nonmodular, and it is well-known that $O_2^+(\F_p)$ can be generated by the following two matrices
\begin{equation}
\label{ }
\upxi:=\begin{pmatrix}
0 & 1  \\
1 & 0
\end{pmatrix},\quad
\uptau_a:=\begin{pmatrix}
a & 0  \\
0 & a^{-1}
\end{pmatrix},
\end{equation}
where $\F_p^{\times}=\ideal{a}$. The special orthogonal group $SO_{2}^+(\F_p)$  is  generated by
$\uptau_a$ and is of order $p-1$.
It follows from \cite[Example 6]{NS02} that 
\begin{equation}
\label{ }
\F_p[V]^{SO_{2}^+(\F_p)}=\F_p[x_1,x_2]^{SO_{2}^+(\F_p)}=\F_p[x_1x_2, x_1^{p-1}, x_2^{p-1}]
\end{equation}
is a hypersurface, subject to the unique relation:  
\begin{equation}
\label{ }
(x_1x_2)^{p-1}=x_1^{p-1}x_2^{p-1}.
\end{equation}

Note that the resulting matrix of each element $g\in O_2^+(\F_p)$ acting on $V^*$ is the inverse of the transpose of $g$.
Thus the resulting matrix of $\upxi$ on $V\oplus V^*$ is
\begin{equation}
\label{ }
\begin{pmatrix}
    \upxi  &  0  \\
    0  &  \upxi
\end{pmatrix}_{4\times 4}
\end{equation}
and the resulting matrix of $\uptau_a$ on $V\oplus V^*$ is $\dia\left\{a,a^{-1},a^{-1},a\right\}.$
Hence, the action of $O_2^+(\F_p)$ on $\F_p[V\oplus V^*]$ can be given by
\begin{equation}
\label{ }
\uptau_a(x_1)=a\cdot x_1,~~ \uptau_a(x_2)=a^{-1}\cdot x_2, ~~\uptau_a(y_1)=a^{-1}\cdot y_1,~~ \uptau_a(y_2)=a\cdot y_2.
\end{equation}
Write $A$ for the $\F_p$-algebra generated by the set $\A$ in (\ref{eq-thm1}).  A direct verification shows that each element in $\A$ is fixed by $\uptau_a$, thus $A\subseteq \F_p[V\oplus V^*]^{SO_2^+(\F_p)}$.

Now we are ready to prove Theorem \ref{thm1}.

\begin{proof}[Proof of Theorem \ref{thm1}]
It suffices to show the \textit{claim} that every element in $\F_p[V\oplus V^*]^{SO_2^+(\F_p)}$ must be in $A$.
We first note that up to a nonzero scalar, $\uptau_a$ fixes each monomial $x_1^ux_2^vy_1^sy_2^t\in\F_p[V\oplus V^*]$, where $u,v,s,t\in\N$. Thus $\F_p[V\oplus V^*]^{SO_2^+(\F_p)}$ can be generated by finitely many monomials. 
We may consider an arbitrary monomial $f=x_1^ux_2^vy_1^sy_2^t$ in
$\F_p[V\oplus V^*]^{SO_2^+(\F_p)}$. Then  
$$x_1^ux_2^vy_1^sy_2^t=f=\uptau_a\cdot f=a^{u+t-v-s}\cdot x_1^ux_2^vy_1^sy_2^t$$
which implies that $a^{u+t-v-s}=1$ and thus $p-1$ divides $u+t-v-s$.

We use induction on the degree of $f$ to prove the claim above. Note that
$\deg(f)=u+t+v+s$.  Clearly, there are no linear invariant polynomials in $\F_p[V\oplus V^*]^{SO_2^+(\F_p)}$.

Suppose that $\deg(f)=2$. According to the partition of $2$: $(1,1)$ and $(2,0)$, there are $10$ possibilities for values of the integer vector $(u,v,s,t)$:
\begin{eqnarray*}
&\{(1,1,0,0),(1,0,1,0),(1,0,0,1),(0,1,1,0),(0,1,0,1),(0,0,1,1)\}&\\
&\{(2,0,0,0), (0,2,0,0),(0,0,2,0),(0,0,0,2)\}.&
\end{eqnarray*}
Note that in the first row above, the first two vectors and the last two vectors give us four invariant monomials:
$\{x_1x_2, y_1y_2, x_1y_1, x_2y_2\}$. The middle two vectors in the first row and the four vectors in the second row above would not produce  invariant monomials unless $p=3$. When $p=3$, the six vectors give rise to the following six invariant monomials:
$$\left\{x_1y_2,x_2y_1,x_1^2,x_2^2,y_1^2,y_2^2\right\}$$
respectively, which are also contained in $\A$.  Thus the claim holds for the case $\deg(f)=2$.

We assume that $\deg(f)\geqslant 3$. If there exists at least one of $\{u,v,s,t\}$ is greater than or equal to $p-1$, without loss of generality, say $u$, then $f=x_1^{p-1}\cdot f'$. As $\deg(f')<\deg(f)$, the induction hypothesis implies that $f'$ can be algebraically expressed  by elements of $\A$. Thus $f\in A$.
Hence, we may assume that $0\leqslant u,v,s,t<p-1$. This also means that
 at least two variables of $\{x_1,x_2,y_1,y_2\}$ are involved in $f$.  Thus, we need to discuss six subcases. 
 
Suppose that $x_1x_2$ divides $f$. We write $f=(x_1x_2)\cdot f'$. Note that $x_1x_2\in\A$, thus the induction hypothesis implies that $f\in A$. Similarly, if $y_1y_2$ (or $x_1y_1, x_2y_2$) divides $f$, as these monomials of degree $2$ are in $\A$, then $f\in A$ by the induction hypothesis.

The remaining two subcases are: $x_1y_2$ divides $f$ or $x_2y_1$ divides $f$. The proofs are similar. In fact, 
assume that $x_1y_2$ divides $f$. If one of $\{v,s\}$ is nonzero, then  
one of $\{x_1x_2, x_1y_1,y_1y_2,x_2y_2\}$ must divide $f$. We have seen that $f\in A$ in the previous paragraph. Thus we only consider the case where $v=s=0$, i.e., $f=x_1^uy_2^t.$
Note that $0< u,t<p-1$, thus $u+t<2(p-1)$. As $p-1$ divides $u+t=\deg(f)\geqslant 3$, it follows that
$u+t=p-1.$  Hence, 
$$f=x_1^{p-1-i}y_2^i$$
for some $i\in\{1,2,\dots,p-1\}$. Similarly, if $x_2y_1$ divides $f$, then
$$f=x_2^{p-1-i}y_1^i$$
for some $i\in\{1,2,\dots,p-1\}$.  Therefore, the claim follows and $\F_p[V\oplus V^*]^{SO_2^+(\F_p)}=A$.
\end{proof}

Recall that $[O_2^+(\F_p): SO_2^+(\F_p)]=2$, thus 
$$\left\{I_2=\begin{pmatrix}
      1& 0   \\
      0&1  
\end{pmatrix}, \upxi=\begin{pmatrix}
    0  & 1   \\
    1  & 0 
\end{pmatrix}\right\}$$
can be chosen as the set of representatives for the left coset of $O_2^+(\F_p)$ over $SO_2^+(\F_p)$. To prove Theorem \ref{thm2}, we will use
Theorem \ref{thm1} and the relative Reynolds operator: 
\begin{equation} \label{R+}
\begin{aligned}
\R^+: & ~~~~\F_p[V\oplus V^*]^{SO_2^+(\F_p)}\ra\F_p[V\oplus V^*]^{O_2^+(\F_p)},\\
&~~~~f\mapsto\frac{1}{[O_2^+(\F_p): SO_2^+(\F_p)]}\sum_{\bar{g}\in O_2^+(\F_p)/SO_2^+(\F_p)}\bar{g}\cdot f
\end{aligned}
\end{equation}
where each $f\in \F_p[V\oplus V^*]^{SO_2^+(\F_p)}$ maps to $\frac{1}{2}(f+\upxi \cdot f).$

Together with Theorem \ref{thm1}, we will apply the relative Reynolds operator $\R^+$ and \cite[Lemma 3.1]{Che18} to give a proof of Theorem \ref{thm2}. Recall that $\F_p[V]^{O_2^+(\F_p)}=\F_p[x_1x_2, x_1^{p-1}+x_2^{p-1}]$ is a polynomial ring; see \cite[Example 5]{NS02}.

\begin{proof}[Proof of Theorem \ref{thm2}]
Let $B$ be the $\F_p$-subalgebra of $\F_p[V\oplus V^*]^{O_2^+(\F_p)}$, generated by the set $\B$ in (\ref{eq-thm2}). The natural embedding 
$\F_p[V\oplus V^*]^{O_2^+(\F_p)}\subseteq \F_p[V\oplus V^*]^{SO_2^+(\F_p)}$ also allows us to regard $B$ as an $\F_p$-subalgebra of $\F_p[V\oplus V^*]^{SO_2^+(\F_p)}$. 
Thus, $\R^+$  is a surjective homomorphism of $\F_p[V\oplus V^*]^{O_2^+(\F_p)}$-modules. 

Note that the four elements $x_1x_2,y_1y_2,x_1^{p-1}+x_2^{p-1}, y_1^{p-1}+y_2^{p-1}$ in $\B$ form a homogeneous system of parameters (hsop for short). In fact, by \cite[Lemma 2.6.3]{CW11}, it suffices to show that the variety determined by these four elements over is algebraic closure $\ol{\F}_p$ is $\{0\}$. Assume that $v=(a_1,a_2,b_1,b_2)$ be an arbitrary point in the variety. Then 
$$a_1a_2=b_1b_2=a_1^{p-1}+a_2^{p-1}=b_1^{p-1}+b_2^{p-1}=0.$$
In particular, $a_2^{p-1}=-a_1^{p-1}.$ Thus $0=(a_1a_2)^{p-1}=a_1^{p-1}a_2^{p-1}=(a_1^{p-1})(-a_1^{p-1})=-a_1^{2p-2}$. This implies that $a_1=0$, as $a_1\in \ol{\F}_p$.  Moreover, $a_2^{p-1}=0$ and $a_2=0$ as well. Similarly, using 
$b_1b_2=0$ and $b_1^{p-1}+b_2^{p-1}=0$ obtains that $b_1=b_2=0$. Hence, $v=0$. This proves that $\B$ contains an hsop and $\F_p[V\oplus V^*]^{SO_2^+(\F_p)}$ is integral over $B$.

Therefore, we may write 
$$\F_p[V\oplus V^*]^{SO_2^+(\F_p)}=B+\sum_{\updelta\in\Delta} \updelta\cdot B$$
where $\Delta \cup\{1\}$ is a homogeneous generating set of $\F_p[V\oplus V^*]^{SO_2^+(\F_p)}$ as a $B$-module. Specifically, by Theorem \ref{thm1},  we may choose
$$\Delta:=\left\{(x_1y_1)^{r_1}(x_2y_2)^{r_2}\cdot \prod_{i=0}^{p-1}(x_1^{p-1-i}y_2^i)^{r_{i,3}}(x_2^{p-1-i}y_1^i)^{r_{i,4}}\mid  0\leqslant r_1,r_2,r_{i,3},r_{i,4}\leqslant m\in\N^+\right\}$$
for some positive integer $m$. Note that we assume that $1\notin\Delta$.

Let $\J$ denote the ideal generated by $\B$ in $\F_p[V\oplus V^*]^{SO_2^+(\F_p)}$.  By \cite[Lemma 3.1]{Che18}, it suffices to prove that $\R^+(\updelta)\in\J$ for all $\updelta\in\Delta$. Suppose that
$$\updelta:=(x_1y_1)^{r_1}(x_2y_2)^{r_2}\cdot \prod_{i=0}^{p-1}(x_1^{p-1-i}y_2^i)^{r_{i,3}}(x_2^{p-1-i}y_1^i)^{r_{i,4}}$$
for some $r_1,r_2,r_{i,3},r_{i,4}\in\N$. Then $\R^+(\updelta)=\frac{1}{2}(\updelta+\upxi \cdot \updelta)$, which can be expressed as 
{\small $$
\frac{1}{2}\left((x_1y_1)^{r_1}(x_2y_2)^{r_2}\prod_{i=0}^{p-1}(x_1^{p-1-i}y_2^i)^{r_{i,3}}(x_2^{p-1-i}y_1^i)^{r_{i,4}}+(x_2y_2)^{r_1}(x_1y_1)^{r_2} \prod_{i=0}^{p-1}(x_2^{p-1-i}y_1^i)^{r_{i,3}}(x_1^{p-1-i}y_2^i)^{r_{i,4}}\right).
$$}
We use induction on the degree of $\updelta$ to prove that $\R^+(\updelta)\in\J.$
Note that $\deg(\updelta)\geqslant 2$.  

Suppose that $\deg(\updelta)=2$ and $p>3$. Then
$\updelta$ is either equal to $x_1y_1$ or $x_2y_2$. If $\updelta=x_1y_1$, then
$\R^+(\updelta)=\frac{1}{2}(x_1y_1+x_2y_2)\in \J$; similarly, if 
$\updelta=x_2y_2$, then
$\R^+(\updelta)=\frac{1}{2}(x_2y_2+x_1y_1)\in \J$ as well.  Moreover, if $p=3$, then
$\updelta$ must be one of 
$$\left\{x_1y_1,x_2y_2, x_1^{2-i}y_2^i, x_2^{2-i}y_1^i\mid 0\leqslant i\leqslant 2\right\}.$$
Clearly, $\R^+(\updelta)\in\J$, when $\updelta\in\{x_1y_1,x_2y_2\}$. We observe that
$$\R^+(x_1^{2-i}y_2^i)=\R^+(x_2^{2-i}y_1^i)=\frac{1}{2}\left(x_1^{2-i}y_2^i+x_2^{2-i}y_1^i\right)\in\J.$$
Thus, the statement holds for the case of degree $2$.

Now we may suppose that $\deg(\updelta)\geqslant 3$. Our arguments can be separated into the following three cases:
\textsc{Case 1:} Both $r_1$ and $r_2$ are positive. Then
$$
\R^+(\updelta)=(x_1x_2)\cdot \R^+(\updelta')\in \J
$$
because $\updelta'\in\Delta$, $\deg(\updelta')<\deg(\updelta)$, and we may use the induction hypothesis. 

\textsc{Case 2:} One of $\{r_1,r_2\}$ is positive and the other one is zero. Without loss of generality, we may assume that
$r_1>0$ and $r_2=0$. If there are two positive numbers in $\{r_{0,3},\dots, r_{p-1,3}, r_{0,4},\dots, r_{p-1,4}\}$, then
$\R^+(\updelta)$ can be written as either $(x_1x_2)\cdot \R^+(\updelta')$ or 
$(y_1y_2)\cdot \R^+(\updelta')$ for some $\updelta'\in\Delta$.  Thus, the induction hypothesis implies that
$\R^+(\updelta)\in \J$. This means that we only need to consider the following subcases:

\textsc{Subcase 2.1:} all the numbers in $\{r_{0,3},\dots, r_{p-1,3}, r_{0,4},\dots, r_{p-1,4}\}$ are zero. Thus
$\updelta=(x_1y_1)^{r_1}$ and 
\begin{eqnarray*}
\R^+(\updelta)&=&\frac{1}{2}\left((x_1y_1)^{r_1}+(x_2y_2)^{r_1}\right)\\
&=&\frac{1}{2}\left((x_1y_1+x_2y_2)^{r_1}-\sum_{i=1}^{r_1-1}{r_1\choose i} (x_1y_1)^{r_1-i}(x_2y_2)^{i}\right).
\end{eqnarray*}
Thus, 
\begin{eqnarray*}
\R^+(\updelta)&=&\R^+\left(\R^+(\updelta)\right)\\
&=&\frac{1}{2}\left((x_1y_1+x_2y_2)^{r_1}-\R^+\left(\sum_{i=1}^{r_1-1}{r_1\choose i} (x_1y_1)^{r_1-i}(x_2y_2)^{i}\right)\right)\\
&=&\frac{1}{2}\left((x_1y_1+x_2y_2)^{r_1}-(x_1x_2\cdot y_1y_2)\left(\sum_{i=0}^{r_1-2}{r_1\choose i+1}\R^+ \left((x_1y_1)^{r_1-2-i}(x_2y_2)^{i}\right)\right)\right)
\end{eqnarray*}
which belongs to $\J$ by the induction hypothesis, because each term $(x_1y_1)^{r_1-2-i}(x_2y_2)^{i}$
has degree lower than $\updelta$. 

\textsc{Subcase 2.2:} One of $\{r_{0,3},\dots, r_{p-1,3}, r_{0,4},\dots, r_{p-1,4}\}$ is positive and all others are zero.
Without loss of generality, we may assume that $r_{0,3}$ is positive and all others are zero. Thus $\updelta=(x_1y_1)^{r_1}\cdot (x_1^{p-1})^{r_{0,3}}$ and $\R^+(\updelta)=\frac{1}{2}\left((x_1y_1)^{r_1}\cdot (x_1^{p-1})^{r_{0,3}}+(x_2y_2)^{r_1}\cdot (x_2^{p-1})^{r_{0,3}}\right)$, which can be expressed as
$$\frac{1}{2}\left[\left((x_1y_1)^{r_1}+(x_2y_2)^{r_1}\right)\cdot \left((x_1^{p-1})^{r_{0,3}}+(x_2^{p-1})^{r_{0,3}}\right)-(x_1y_1)^{r_1}\cdot (x_2^{p-1})^{r_{0,3}}-(x_2y_2)^{r_1}\cdot (x_1^{p-1})^{r_{0,3}}\right].
$$
We have seen in the previous subcase that  $\frac{1}{2}\left((x_1y_1)^{r_1}+(x_2y_2)^{r_1}\right)\in\J$, 
it suffices to show that
$$\frac{1}{2}\left[(x_1y_1)^{r_1}\cdot (x_2^{p-1})^{r_{0,3}}+(x_2y_2)^{r_1}\cdot (x_1^{p-1})^{r_{0,3}}\right]\in\J.$$
To see that, we define 
\begin{equation}
\label{ }
\ell:=|r_1-(p-1)\cdot r_{0,3}|.
\end{equation}
If $r_1\leqslant (p-1)\cdot r_{0,3}$, then
\begin{eqnarray*}
\frac{1}{2}\left[(x_1y_1)^{r_1}\cdot (x_2^{p-1})^{r_{0,3}}+(x_2y_2)^{r_1}\cdot (x_1^{p-1})^{r_{0,3}}\right]
&=& \frac{1}{2}\cdot (x_1x_2)^{r_1}\left[y_1^{r_1}\cdot x_2^{\ell}+y_2^{r_1}\cdot x_1^{\ell}\right]\\
&=&(x_1x_2)^{r_1} \cdot \R^+(y_1^{r_1}\cdot x_2^{\ell})\in\J
\end{eqnarray*}
by the induction hypothesis. Similarly, if $r_1>(p-1)\cdot r_{0,3}$, then
\begin{eqnarray*}
\frac{1}{2}\left[(x_1y_1)^{r_1}\cdot (x_2^{p-1})^{r_{0,3}}+(x_2y_2)^{r_1}\cdot (x_1^{p-1})^{r_{0,3}}\right]
&=& \frac{1}{2}\cdot (x_1x_2)^{(p-1)\cdot r_{0,3}}\left[x_1^\ell\cdot y_1^{r_1}+x_2^{\ell}\cdot y_2^{r_1}\right]\\
&=&(x_1x_2)^{(p-1)\cdot r_{0,3}} \cdot \R^+(x_1^\ell\cdot y_1^{r_1})\in\J
\end{eqnarray*}
as well.

\textsc{Case 3:} Both $r_1$ and $r_2$ are zero. Then 
$$\updelta=\prod_{i=0}^{p-1}(x_1^{p-1-i}y_2^i)^{r_{i,3}}(x_2^{p-1-i}y_1^i)^{r_{i,4}}=x_1^ax_2^by_1^cy_2^d$$
for some $a,b,c,d\in\N$.
Note that $x_1x_2$ and $y_1y_2$ belong to $\B$. Thus if $x_1x_2$ (or $y_1y_2$) divides $\updelta$, then 
$\R^+(\updelta)=x_1x_2\cdot \R^+(\updelta')$ (or $=y_1y_2\cdot \R^+(\updelta')$) for some $\updelta'\in\Delta$ with
$\deg(\updelta')<\deg(\updelta)$. Thus $\R^+(\updelta)\in\J$ by the induction hypothesis. 

Now we may assume that neither $x_1x_2$ nor $y_1y_2$ divides $\updelta$. This means that $\updelta$ involves at most two variables. Thus $\updelta$ can be expressed as one of the following 
$$\{x_1^ay_1^c, x_1^ay_2^d, x_2^by_1^c, x_2^by_2^d\}.$$
Note that $\updelta\neq 1$. Thus one  of $\{r_{0,3},\dots, r_{p-1,3}, r_{0,4},\dots, r_{p-1,4}\}$ must be positive. 

\textsc{Subcase 3.1:} Two and more of $\{r_{0,3},\dots, r_{p-1,3}, r_{0,4},\dots, r_{p-1,4}\}$ are positive. In this subcase,
two variables are involved in $\updelta$. (1) If $\updelta=x_1^ay_1^c$, 
then $a,c\in\N^+$. Note that $\deg(\updelta)\geqslant 2(p-1)\geqslant 4$, thus either $a$ or $c$ is greater than or equal to $2$. 
We may assume that $a\geqslant 2$.  Let $f:=x_1y_1+x_2y_2\in\B$. Then
$\updelta = x_1^ay_1^c=x_1y_1\cdot x_1^{a-1}y_1^{c-1} = (f-x_2y_2) \cdot x_1^{a-1}y_1^{c-1}=
f\cdot x_1^{a-1}y_1^{c-1}-x_2y_2 \cdot x_1^{a-1}y_1^{c-1}$. Thus
$$\R^+(\updelta)=f\cdot \R^+(x_1^{a-1}y_1^{c-1})-x_1x_2\cdot \R^+(y_2 \cdot x_1^{a-2}y_1^{c-1})\in\J$$
because $\R^+(x_1^{a-1}y_1^{c-1})$ and $\R^+(y_2 \cdot x_1^{a-2}y_1^{c-1})$ both lie in $\J$ by the induction hypothesis. 
(2) If $\updelta=x_1^ay_2^d$,  then $a,d\in\N^+$.  Since $\deg(\updelta)\geqslant 2(p-1)$, it follows that either $a$ or $d$ must be greater than or equal to $p-1$. We may assume that $a\geqslant p-1$.  Let $h:=x_1^{p-2}y_2+x_2^{p-2}y_1\in\B$. Then
$\updelta=x_1^{p-2}y_2\cdot x_1^{a-p+2}y_2^{d-1}=(h-x_2^{p-2}y_1)\cdot x_1^{a-p+2}y_2^{d-1}=
h \cdot x_1^{a-p+2}y_2^{d-1}- x_2^{p-2}y_1\cdot x_1^{a-p+2}y_2^{d-1}$. By the induction hypothesis, we see that
$h \cdot \R^+(x_1^{a-p+2}y_2^{d-1})\in \J$ and 
$$\R^+(x_2^{p-2}y_1\cdot x_1^{a-p+2}y_2^{d-1})=x_1x_2\cdot \R^+(x_1^{a-p+1}x_2^{p-3}y_1y_2^{d-1})\in \J.$$
Hence, $\R^+(\updelta)\in\J$ as well.  (3) Note that $x_1$ and $x_2$ are symmetric, thus if $\updelta=x_2^by_1^c$ or $x_2^by_2^d$, then a similar argument can be applied and the same conclusion will be obtained. 

\textsc{Subcase 3.2:} One of $\{r_{0,3},\dots, r_{p-1,3}, r_{0,4},\dots, r_{p-1,4}\}$ is positive and others are zero. 
We may assume that $r_{i,3}$ is positive for some $i\in\{0,1,\dots,p-1\}$. Symmetrically, the following argument also works for the case where some $r_{i,4}$ is positive, and obtains the same conclusion. Thus, let us focus on the case that
$r_{i,3}>0$. In this subcase, $\updelta=(x_1^{p-1-i}y_2^i)^{r_{i,3}}$ and 
\begin{eqnarray*}
\R^+(\updelta)&=&\frac{1}{2}\left((x_1^{p-1-i}y_2^i)^{r_{i,3}}+(x_2^{p-1-i}y_1^i)^{r_{i,3}}\right) \\
&=&\frac{1}{2}\left((x_1^{p-1-i}y_2^i+x_2^{p-1-i}y_1^i)^{r_{i,3}}-\sum_{j=1}^{r_{i,3}-1}{r_{i,3}\choose j} (x_1^{p-1-i}y_2^i)^{r_{i,3}-j}(x_2^{p-1-i}y_1^i)^{j}\right).
\end{eqnarray*}
Note that $x_1^{p-1-i}y_2^i+x_2^{p-1-i}y_1^i\in\B$ and $\R^+(\updelta)=\R^+\left(\R^+(\updelta)\right)$. Thus, it suffices to show that 
$$\R^+\left((x_1^{p-1-i}y_2^i)^{r_{i,3}-j}(x_2^{p-1-i}y_1^i)^{j}\right)\in\J$$
for all $j=1,2,\dots,r_{i,3}-1$. For notational simplicity, we write  $f_{ij}:=(x_1^{p-1-i}y_2^i)^{r_{i,3}-j}(x_2^{p-1-i}y_1^i)^{j}.$
When $i=0,1,\dots,p-2$, we observe that $x_1x_2\in\B$ and $x_1x_2$ divides $f_{ij}$. Thus
$$\R^+(f_{ij})=x_1x_2\cdot\R^+(f_{ij}/(x_1x_2))\in\J$$
because $f_{ij}/(x_1x_2)$ has degree less than the degree of $\updelta$ and the induction hypothesis applies. If $i=p-1$, then
$y_1y_2$ divides $f_{ij}$. As $y_1y_2$ belongs to $\B$ as well, the same reason implies that
$$\R^+(f_{ij})=y_1y_2\cdot\R^+(f_{ij}/(y_1y_2))\in\J.$$
Hence, $\R^+(\updelta)\in\J$ in this subcase.

The arguments of three cases above complete the proof and therefore, $\F_p[V\oplus V^*]^{O_2^{+}(\F_p)}$ can be  generated by the homogeneous set $\B$.
\end{proof}

\section{$O_2^-(\F_p)$-Invariants} \label{sec3}
\setcounter{equation}{0}
\renewcommand{\theequation}
{3.\arabic{equation}}
\setcounter{theorem}{0}
\renewcommand{\thetheorem}
{3.\arabic{theorem}}

\noindent In this section, we will use a different method that comes from covariant theory of finite groups to compute 
$\F_p[V\oplus V^*]^{O_2^-(\F_p)}$. Note that $|O_2^-(\F_p)|=2(p+1)$ for $p\geqslant 3$, thus the standard representation $V$ is also non-modular, but generators of $O_2^-(\F_p)$ depend on the prime $p$. In fact, by \cite[Example 5]{NS02}, we know that if $p\equiv 3 \mod 4$, then $O_2^-(\F_p)$ is generated by the following matrices 
\begin{equation}
\label{ }
\upeta :=\begin{pmatrix}
-1 & 0  \\
0 & 1
\end{pmatrix}, \quad
\upsigma_1 :=\begin{pmatrix}
a & -b  \\
b & a
\end{pmatrix},
\end{equation}
where $a^2+b^2=1 \text{~with~} a, b\in \F_p^{\times}$.
The invariant ring 
\begin{equation}
\label{ }
\F_p[V]^{O_2^-(\F_p)}=\F_p[x_1^2+x_2^2, x_1^{p+1}+x_2^{p+1}]
\end{equation}
is a polynomial ring. If $p\equiv 1 \mod 4$, then we may choose $\lambda$ as a generator of $\F_p^{\times}$ and $O_{2}^-(\F_p)$ can be generated by $\upeta$ and 
\begin{equation}
\label{ }
\upsigma_2 :=\begin{pmatrix}
a & \lambda^{-1}\cdot b  \\
b & a
\end{pmatrix},
\end{equation}
where $a^2-\lambda^{-1}\cdot b^2=1, a, b\in \F_p^{\times}$. Moreover, 
$\F_p[V]^{O_2^-(\F_p)}=\F_p[x_1^2-\lambda\cdot x_2^2, x_1^{p+1}-\lambda\cdot x_2^{p+1}]$
is a polynomial ring as well.

This section will be devoted to giving a detailed proof of Theorem \ref{thm3} for the case $p\equiv 3 \mod 4$, and 
the case $p\equiv 1 \mod 4$ can be verified in the same way.  Thus, throughout the rest of this section, we assume that
$p\equiv 3 \mod 4$, and for simplicity, we write
\begin{equation}
\label{ }
G:=O_2^-(\F_p)\times O_2^-(\F_p)
\end{equation}
for the direct product of two copies of $O_2^-(\F_p)$. Note that $O_2^-(\F_p)$ can be regarded as a subgroup of $G$ via the standard embedding, and $G$ also acts on 
$\F_p[V\oplus V^*]$ in the natural way.

\subsection{Hilbert series and $s$-invariants}

Consider the invariant ring $\F_p[V\oplus V^*]^G$. Note that $$\{x_1^2+x_2^2, x_1^{p+1}+x_2^{p+1}, y_1^2+y_2^2, y_1^{p+1}+y_2^{p+1}\}$$ is a homogeneous system of parameters for $\F_p[V\oplus V^*]^G$, and the product of their degrees is equal to the order of $G$, thus it follows from \cite[Corollary 3.1.6]{CW11} that 
$$\F_p[V\oplus V^*]^G=\F_p[x_1^2+x_2^2, x_1^{p+1}+x_2^{p+1}, y_1^2+y_2^2, y_1^{p+1}+y_2^{p+1}]$$
is a polynomial algebra.  Thus the Hilbert series of $\F_p[V\oplus V^*]^G$ is
\begin{equation}\label{eq3.5}
\mathcal{H}(\F_p[V\oplus V^*]^G; t)=\frac{1}{(1-t^2)^2(1-t^{p+1})^2}.
\end{equation}
Choose $\F_p[V\oplus V^*]^G$ as a Noether normalization of $\F_p[V\oplus V^*]^{O_2^-(\F_p)}$. Since 
$\F_p[V\oplus V^*]^{O_2^-(\F_p)}$ is Cohen-Macaulay, it follows from \cite[Corollary 3.1.4]{CW11}  that
 $\F_p[V\oplus V^*]^{O_2^-(\F_p)}$ is a free module of rank 
 $$[G:O_2^-(\F_p)]=\frac{|G|}{|O_2^-(\F_p)|}=|O_2^-(\F_p)|=2(p+1)$$
over $\F_p[V\oplus V^*]^G$. Hence,  the Hilbert series of $\F_p[V\oplus V^*]^{O_2^-(\F_p)}$ can be written as
\begin{equation}\label{eq3.6}
\mathcal{H}\left(\F_p[V\oplus V^*]^{O_2^-(\F_p)}; t\right)=\frac{1+t^{s_1}+\dots+t^{s_{2p+1}}}{(1-t^2)^2(1-t^{p+1})^2}
\end{equation}
for some $s_1\leqslant \cdots \leqslant s_{2p+1}\in\N^+$.

In the language of modules of covariants (see \cite[Section 4]{BC10}), the invariant ring $\F_p[V\oplus V^*]^{O_2^-(\F_p)}$ is isomorphic to the module of covariants 
$$\F_p[V\oplus V^*]^G(M)$$
where $M:=\Ind_{O_2^-(\F_p)}^G \F_p$ denotes the permutation $\F_pG$-module on the left coset space $G/O_2^-(\F_p)$ with dimension $|O_2^-(\F_p)|=2(p+1)$. Thus the quotient of two Hilbert series (\ref{eq3.5}) and (\ref{eq3.6}) is
\begin{equation}
\label{eq3.7}
\frac{\mathcal{H}(\F_p[V\oplus V^*]^G(M); t)}{\mathcal{H}(\F_p[V\oplus V^*]^G; t)}=\frac{\mathcal{H}(\F_p[V\oplus V^*]^{O_2^-(\F_p)}; t)}{\mathcal{H}(\F_p[V\oplus V^*]^G; t)}=1+t^{s_1}+\dots+t^{s_{2p+1}}.
\end{equation}
Recall that the $s$-invariant of $\F_p[V\oplus V^*]^G(M)$ also appears in this quotient; see \cite[Introduction, page 2]{BC10}. More precisely, 
\begin{equation}
\label{eq3.8}
\frac{\mathcal{H}(\F_p[V\oplus V^*]^{O_2^-(\F_p)}; t)}{\mathcal{H}(\F_p[V\oplus V^*]^G; t)}=
r_{O_2^-(\F_p)}+s_{O_2^-(\F_p)}(t-1)+O\left((t-1)^2\right)
\end{equation}
where $r_{O_2^-(\F_p)}:=r_{\F_p[V\oplus V^*]^G}(\F_p[V\oplus V^*]^{O_2^-(\F_p)})=\dim(M)=2(p+1)$ and 
$$s_{O_2^-(\F_p)}:=s_{\F_p[V\oplus V^*]^G}(\F_p[V\oplus V^*]^{O_2^-(\F_p)})$$ denotes the $s$-invariant of $\F_p[V\oplus V^*]^{O_2^-(\F_p)}$ over
$\F_p[V\oplus V^*]^G$.

\begin{lem}\label{lem3.1}
$s_{O_2^-(\F_p)}=2(p+1)^2.$
\end{lem}

\begin{proof}
Together (\ref{eq3.7}) and (\ref{eq3.8}) imply that  
\begin{equation}\label{eq3.9}
1+t^{s_1}+\dots+t^{s_{2p+1}}=r_{O_2^-(\F_p)}+s_{O_2^-(\F_p)}(t-1)+O\left((t-1)^2\right).
\end{equation}
By \cite[Proposition 3.1.4]{NS02}, the Laurent expansion of  $\mathcal{H}(\F_p[V\oplus V^*]^{O_2^-(\F_p)}; t)$ gives us  
$$
\frac{1+t^{s_1}+\dots+t^{s_{2p+1}}}{(1-t^2)^2\left(1-t^{p+1}\right)^2}=\mathcal{H}(\F_p[V\oplus V^*]^{O_2^-(\F_p)}; t)=\frac{1}{|O_2^-(\F_p)|}\left(\frac{1}{(1-t)^4}+\frac{c}{(1-t)^3}+\cdots\right),
$$
where $2\cdot c$ equals the number of all reflections of $O_2^-(\F_p)$ on $V\oplus V^*$. 
Multiplying both sides by $(1-t)^4$, we obtain
$$
\frac{1+t^{s_1}+\dots+t^{s_{2p+1}}}{(1+t)^2\left(1+t+\dots +t^p\right)^2}=\frac{1}{|O_2^-(\F_p)|}\left(1+c(1-t)+\cdots\right).
$$
This equation together with (\ref{eq3.9}) implies that 
$$
\frac{|O_2^-(\F_p)|}{(1+t)^2\left(1+t+\dots +t^p\right)^2}=\frac{1+c(1-t)+\cdots}{r_{O_2^-(\F_p)}+s_{O_2^-(\F_p)}(t-1)+O\left((t-1)^2\right)}.
$$
Differentiating with respect to $t$ by logarithmic differentiation, we obtain 
\begin{eqnarray*}
&&|O_2^-(\F_p)|\cdot\left((-2)\frac{(1+t)\left(\sum_{j=1}^{p+1}t^{j-1}\right)^2+(1+t)^2\sum_{j=1}^{p+1}t^{j-1}(1+2t+\dots+pt^{p-1})}{(1+t)^4\left(\sum_{j=1}^{p+1}t^{j-1}\right)^4}\right)\\
&=&\frac{\left(-c+\cdots\right)\cdot \left(r_{O_2^-(\F_p)}+s_{O_2^-(\F_p)}(t-1)+O\left((t-1)^2\right)\right)}{\left(r_{O_2^-(\F_p)}+s_{O_2^-(\F_p)}(t-1)+O\left((t-1)^2\right)\right)^2}\\
&&-\frac{\left(1+c(1-t)+\cdots\right)\cdot\left(s_{O_2^-(\F_p)}+O\left((t-1)\right)\right)}{\left(r_{O_2^-(\F_p)}+s_{O_2^-(\F_p)}(t-1)+O\left((t-1)^2\right)\right)^2}.
\end{eqnarray*}
Setting $t=1$ gives 
\begin{equation}
\label{eq3.10}
-|O_2^-(\F_p)|\cdot\frac{2^2(p+1)^2+2^2(p+1)^2p}{2^4(p+1)^4}=\frac{-c\cdot r_{O_2^-(\F_p)}-s_{O_2^-(\F_p)}}{(r_{O_2^-(\F_p)})^2}.
\end{equation}
Note that $c=0$, because the image of $O_2^-(\F_p)$ on $V\oplus V^*$ contains no reflections (see \cite[Theorem 1]{CW17}). Recall that $r_{O_2^-(\F_p)}=2(p+1)$, thus substituting these numbers back to (\ref{eq3.10}), we conclude that
$s_{O_2^-(\F_p)}=2(p+1)\cdot(p+1)=2(p+1)^2.$
\end{proof}

\subsection{Jacobian criterion}

For $0\leqslant i\leqslant p+1$ and $1\leqslant j\leqslant p$, we define $u:=x_1y_1+x_2 y_2$ and
\begin{eqnarray*}
f_i &:= & u^{i} \\
f_{p+1+j} &:= & \Tr(x_1^{p+1-j}y_1^j)
\end{eqnarray*}
where $\Tr:=\Tr^{O_2^-(\F_p)}$ denotes the trace map. Clearly, Theorem \ref{thm3} is a direct consequence of the following 
theorem. Thus, the rest of this section is devoted to proving Theorem \ref{thm4}.

\begin{thm}\label{thm4}
As a free $\F_p[V\oplus V^*]^G$-module, $\F_p[V\oplus V^*]^{O_2^-(\F_p)}$ has a basis 
$\{f_0,f_1,\dots,f_{2p+1}\}$.
\end{thm}

\begin{proof}
By \cite[Lemma 5]{BC10},  there exists an $\F_p[V\oplus V^*]^G$-module isomorphism: 
\begin{equation}
\label{ }
\uppsi: \F_p[V\oplus V^*]^{O_2^-(\F_p)} \ra \F_p[V\oplus V^*]^G(M)
\end{equation}
defined by 
$f_j\mapsto \Tr_{O_2^-(\F_p)}^G(f_j\otimes 1)=\sum\limits_{g\in G/O_2^-(\F_p)}g(f_j\otimes 1)=\sum\limits_{g\in G/O_2^-(\F_p)}g(f_j)\otimes g$. 

For $j=0,1,\dots, 2p+1$, we denote by $\upomega_j:=\uppsi(f_j)$. 
Hence, to show that $\{f_0,f_1,\dots,f_{2p+1}\}$ is a free basis of $\F_p[V\oplus V^*]^{O_2^-(\F_p)}$ over $\F_p[V\oplus V^*]^G$, it suffices to show that $\{\upomega_0, \upomega_1,\dots,\upomega_{2p+1}\}$ is a free basis of 
$\F_p[V\oplus V^*]^G(M)$ as an $\F_p[V\oplus V^*]^G$-module. We will use the Jacobian criterion in \cite[Theorem 3 (iii)]{BC10} to prove this statement.

Note that the action of $G$ is degree-preserving, thus $\deg(\upomega_j)=\deg(f_j)$, and 
$$\sum_{j=0}^{2p+1}\deg(\upomega_j)=\sum_{j=0}^{2p+1}\deg(f_j)=2(p+1)^2$$
which is equal to the $s$-invariant by Lemma \ref{lem3.1}.  This fact, together with Lemma \ref{lem3.3} below, shows that
$\{\upomega_0, \upomega_1,\dots,\upomega_{2p+1}\}$ is a free basis of $\F_p[V\oplus V^*]^G(M)$ over $\F_p[V\oplus V^*]^G$. Therefore, $\{f_0,f_1,\dots,f_{2p+1}\}$ is a free basis of $\F_p[V\oplus V^*]^{O_2^-(\F_p)}$ as an $\F_p[V\oplus V^*]^G$-module.
\end{proof}

We take the definition of Jacobian determinant of covariants in \cite[Section 3.2]{BC10}. To complete the proof of 
Theorem \ref{thm4}, we need to prove that the Jacobian determinant of 
$\{\upomega_0, \upomega_1,\dots,\upomega_{2p+1}\}$ is nonzero. 

\begin{lem}\label{lem3.3}
$\Jac(\upomega_0, \upomega_1,\dots,\upomega_{2p+1})\neq 0$.
\end{lem}

\begin{proof}
We may take $\{g_i:=(\upsigma_1^i, 1), g_{p+1+i}:=(\upeta\upsigma_1^i, 1), i=0,1,\dots, p\}$ as the set of representatives of the left coset $G/O_2^-(\F_p)$. Thus $\Jac(\upomega_0, \upomega_1,\dots,\upomega_{2p+1})=\det\left(g_i(f_j)\right)_{0\leqslant i,j\leqslant 2p+1}.$
To show this determinant is nonzero, we may endow $\F_p[V\oplus V^*]$ with the lexicographic monomial ordering ($x_1>y_1>x_2>y_2$) and  only need to show that  the following determinant 
$$J:=\det\left(\LT(g_i(f_j))\right)_{0\leqslant i,j\leqslant 2p+1}\neq 0.$$

We use $\upalpha\approx \upbeta$ to denote that there exists a nonzero scalar $c\in\F_p$ such that $\upbeta=c\cdot \upalpha$. By the action of $O_2^-(\F_p)$ on $\F_p[V\oplus V^*]$, it follows that
$$J\approx \det\left(\begin{array}{ccccccc}
1& x_1y_1&\dots& (x_1y_1)^{p+1}&x_1^py_1&\dots&x_1y_1^p \\
1& a_{11}x_1y_1& \dots& (a_{11}x_1y_1)^{p+1}&b_{11}x_1^py_1&\dots&b_{1p}x_1y_1^p \\
\vdots & \vdots & \ddots & \vdots &\vdots&\ddots&\vdots\\
1&a_{p1}x_1y_1&\dots& (a_{p1}x_1y_1)^{p+1}&b_{p1}x_1^py_1&\dots&b_{pp}x_1y_1^p \\
1&-x_1y_1&\dots&(x_1y_1)^{p+1}&-x_1^py_1&\dots&-x_1y_1^p\\
1&-a_{11}x_1y_1& 
\dots&(a_{11}x_1y_1)^{p+1}&-b_{11}x_1^py_1&\dots&-b_{1p}x_1y_1^p \\\vdots & \vdots & \ddots & \vdots &\vdots&\ddots&\vdots\\
1&-a_{p1}x_1y_1&\dots& (a_{p1}x_1y_1)^{p+1}&-b_{p1}x_1^py_1&\dots&-b_{pp}x_1y_1^p\\\end{array}\right)
$$
for some $a_{ij},b_{ij}\in\F_p^\times$, where 
\begin{equation}\label{eq3.12}
\left\{\begin{array}{lll}
a_{0j}&=&1, \quad 1\leqslant j\leqslant p+1,\\
a_{ij}&=&(a_{i1})^j, \quad  1\leqslant i\leqslant p, 2\leqslant j\leqslant p+1,\\
a_{p+1+i,j}&=&(-1)^ja_{ij}, \quad  0\leqslant i\leqslant p, 1\leqslant j\leqslant p+1,\\
b_{0j}&=&1,\quad 1\leqslant j\leqslant p,\\
b_{p+1+i,j}&=&(-1)^jb_{ij},\quad 0\leqslant i\leqslant p, 1\leqslant j\leqslant p.
\end{array}\right.
\end{equation}

Taking out the common factor of each column of the matrix above, we see that $J\neq 0$ if and only if the following matrix 
$$
K:=\left(\begin{array}{ccccccccccc}
1&1&1&\dots&1&1&1&1&\dots&1&1\\
1&a_{11}&a_{11}^2&\dots&a_{11}^p&a_{11}^{p+1}& b_{11}&b_{12}&\dots&b_{1, p-1}&b_{1p} \\
\vdots & \vdots & \vdots & \ddots & \vdots & \vdots &\vdots & \vdots & \ddots & \vdots & \vdots\\
1& a_{p1}&a_{p1}^2&\dots&a_{p1}^p&a_{p1}^{p+1}&b_{p1}&b_{p2}&\dots&b_{p,p-1}&b_{pp} \\
1&-1&1&\dots&-1&1& -1&1&\dots&1&-1\\
1&-a_{11}&a_{11}^2&\dots&-a_{11}^p&a_{11}^{p+1}&-b_{11}&b_{12}&\dots&b_{1,p-1}&-b_{1p}\\
\vdots & \vdots & \vdots & \ddots & \vdots & \vdots &\vdots & \vdots & \ddots & \vdots & \vdots\\
1&-a_{p1}&a_{p1}^2&\dots&-a_{p1}^p&a_{p1}^{p+1}&-b_{p1}&b_{p2}&\dots&b_{p,p-1}&-b_{pp}\\
\end{array}\right)
$$
is invertible. Using Gaussian elimination, together with (\ref{eq3.12}), we see that
$$K\approx \begin{pmatrix}
   K_1   &   K_0 \\
    0  &  K_2
\end{pmatrix}$$
where $K_1$ and $K_2$ are both invertible matrices of size $(p+1)\times (p+1)$. This means that $\det(K)\neq 0$ and therefore, $J$ is nonzero.
\end{proof}

\section*{Acknowledgments}
\noindent 
The symbolic computation language MAGMA \cite{BCP97} (http://magma.maths.usyd.edu.au/) was very helpful. This research was partially supported by NNSF of China under grant No. 12561003.
We thank the anonymous referee of an earlier version of this paper for their careful reading and helpful suggestions. We also thank Yin Chen for valuable comments and many stimulating discussions on finite classical groups.


\raggedright
\end{document}